\documentclass[letterpaper,11pt]{article}
 
\usepackage{amsmath,amsthm}
\usepackage{amsfonts}
\usepackage{latexsym}
\usepackage{amssymb,bm}
\usepackage{mathrsfs}
\usepackage{mathtools}
\usepackage[page]{appendix}
\usepackage{float}
\usepackage{changes}

\usepackage{graphicx}
\usepackage{epstopdf}

 \usepackage{grffile} 

\usepackage{changes}

\usepackage{color}
\usepackage{color}

\newtheorem{theorem}{Theorem}[section]

\newtheorem{proposition}[theorem]{Proposition}

\newtheorem{algorithm}{Algorithm}

\newcommand{\inner}[2]{\langle #1,#2\rangle}

\newcommand{\norm}[1]{\|{#1}\|}

\newcommand{\tos}{\rightrightarrows}

\newcommand{\comenta}[1]{}

\newcommand{\HH}{\mathcal{H}}

\newcommand{\mgap}{\vspace{.1in}}

\begin{document}

\title{On the convergence rate of the scaled proximal decomposition on the graph of a maximal
monotone operator (SPDG) algorithm}
\author{
    M. Marques Alves
\thanks{
Departamento de Matem\'atica,
Universidade Federal de Santa Catarina,
Florian\'opolis, Brazil, 88040-900 ({\tt maicon.alves@ufsc.br}).
The work of this author was partially supported by CNPq grants no.
406250/2013-8, 306317/2014-1 and 405214/2016-2.}
\and
Samara Costa Lima
\thanks{
Departamento de Matem\'atica,
Universidade Federal de Santa Catarina,
Florian\'opolis, Brazil, 88040-900 ({\tt samaraclim@outlook.com}).
The work of this author was partially supported by CAPES Scholarship  no. 201302180.
}
}


\maketitle

\begin{abstract}
Relying on fixed point techniques, Mahey, Oualibouch and Tao  
introduced in \cite{mah.qua.tao-prox.siam95}
the scaled proximal decomposition on the graph of a maximal monotone  operator (SPDG) algorithm and analyzed its performance on inclusions for strongly monotone and Lipschitz continuous operators. The SPDG algorithm generalizes the Spingarn's partial inverse method by allowing scaling factors,  a key strategy to speed up the convergence of numerical algorithms. 

In this note, we show that the SPDG algorithm can alternatively be analyzed by means of the original Spingarn's partial inverse framework, tracing back to the 1983 Spingarn's paper. 
We simply show that under the assumptions considered in \cite{mah.qua.tao-prox.siam95}, the Spingarn's partial inverse 
of the underlying maximal monotone operator is strongly monotone,  which allows 
one to employ recent results on the convergence and iteration-complexity of proximal point type methods for strongly monotone operators. By doing this,  we additionally obtain a potentially faster convergence for the SPDG algorithm and a more accurate upper bound on the number of iterations needed to achieve prescribed tolerances,  specially on ill-conditioned problems.  
\end{abstract}

\pagestyle{plain}


\section*{Introduction}
In his 1983 paper, J. E. Spingarn~\cite{spi-par.amo83} posed the problem of finding $x,u\in \HH$ such that
\begin{align}
 \label{eq:pi}
 x\in V, \quad u\in V^\perp\;\;\mbox{and}\;\; u\in T(x)
\end{align}
where $T$ is a (point-to-set) maximal monotone operator on the Hilbert space $\HH$ and $V$ is a closed vector subspace. 
Problem \eqref{eq:pi} encompasses different problems in applied mathematics and optimization including  minimization of convex functions over subspaces and decomposition problems for the sum of finitely 
many maximal monotone operators (see, e.g., \cite{com-pri.ol14,bur.sag.sch-ine.oms06,ouo-eps.mp04,spi-par.amo83}). 
Spingarn also proposed and studied in~\cite{spi-par.amo83} a proximal point type method -- the Spingarn's partial inverse method -- for solving \eqref{eq:pi}. This is an iterative scheme, which consists essentially of the Rockafellar's proximal point (PP) method  \cite{roc-mon.sjco76} for solving
the inclusion problem defined by the partial inverse of the underlying maximal monotone operator.
Since then, the Spingarn's partial inverse method has been used by many authors as a framework for the design
and analysis of different practical algorithms in optimization and related applications.

One of the problems that arises in Spingarn's approach is the difficulty of dealing with scaling, a key strategy to speed up the 
convergence rate of numerical algorithms.  Relying on fixed point techniques,  Mahey, Oualibouch and 
Tao \cite{mah.qua.tao-prox.siam95} circumvented this problem by 
proposing and analyzing the scaled proximal decomposition on the graph of a maximal monotone 
operator (SPDG) algorithm for solving \eqref{eq:pi}.
In contrast to the Spingarn's partial inverse method, the SPDG algorithm allows
for a scaling of the problem, for which the impact on speeding up its convergence rate 
was observed in \cite{mah.qua.tao-prox.siam95} on ill-conditioned quadratic programming 
problems. 
The analysis of the SPDG algorithm presented in \cite{mah.qua.tao-prox.siam95} consists
of reformulating it as a fixed point iteration  for a certain fixed point operator. This differs
from the original Spingarn's approach in~\cite{spi-par.amo83}, which is supported on the partial inverse of
monotone operators, a concept introduced and 
coined by Spingarn himself. 
Moreover, one of the main contributions of Mahey, Oualibouch and 
Tao is the convergence rate analysis of the SPDG algorithm under strong monotonicity.
In \cite[Theorem 4.2]{mah.qua.tao-prox.siam95}, they proved that
if the maximal monotone operator $T$ in \eqref{eq:pi} is strongly maximal monotone and Lipschitz continuous, then 
the SPDG algorithm converges  at a linear rate to the (unique, in this case) solution of \eqref{eq:pi}.  

In this note, we show that the SPDG algorithm can  alternatively be analyzed within 
the Spingarn's partial inverse framework, tracing back to the original Spingarn's approach. 
We simply show that under the assumptions considered in \cite{mah.qua.tao-prox.siam95}, the Spingarn's partial inverse 
of the underlying maximal monotone operator is strongly monotone,  which allows 
one to employ recent results on the convergence and iteration-complexity of proximal point type methods for strongly monotone operators.
By doing this,  we additionally obtain a potentially faster convergence for the SPDG algorithm and a more accurate upper bound on the number of iterations needed to achieve prescribed tolerances,  specially on ill-conditioned problems.

\mgap
\noindent
This work is organized as follows. Section \ref{sec:pre} presents the basic notation 
to be used in this note and some preliminaries results.
In Section \ref{sec:str}, we present our main contributions, namely the convergence
rate analysis of the SPDG algorithm within the Spingarn's partial inverse framework. Section \ref{sec:cr} is devoted
to some concluding remarks.


\section{Notation and basic results}
\label{sec:pre}

Let $\HH$ be a real Hilbert space with
inner product $\inner{\cdot}{\cdot}$ and induced norm $\|\cdot\|=\sqrt{\inner{\cdot}{\cdot}\textbf{}}$. 
A set-valued map $T:\HH\tos \HH$ is said to be a \emph{monotone operator} if
$\inner{z-z'}{v-v'}\geq 0$ for all $v\in T(z)$ and $v'\in T(z')$. On the other hand, $T:\HH\tos \HH$ is 
\emph{maximal monotone} if $T$ is monotone and its graph 
$G(T):=\{(z,v)\in \HH\times \HH\,|\,v\in T(z)\}$ is not properly contained in the graph of any other 
monotone operator on $\HH$. The inverse of $T:\HH\tos \HH$ is $T^{-1}:\HH\tos \HH$, defined at any
$z\in \HH$ by $v\in T^{-1}(z)$ if and only if $z\in T(v)$. The resolvent of a maximal monotone operator
$T:\HH\tos \HH$ is $(T+I)^{-1}$ and $z=(T+I)^{-1}x$ if and only if $x-z\in T(z)$. The operator 
$\gamma T:\HH\tos \HH$, where $\gamma>0$, is defined by $(\gamma T)z:=\gamma T(z):=\{\gamma v\,|\,v\in T(z)\}$.
A maximal monotone operator $T:\HH\tos \HH$ is said to be \emph{$\mu$-strongly monotone} if $\mu>0$
and $\inner{z-z'}{v-v'}\geq \mu\norm{z-z'}^2$ for all $v\in T(z)$ and $v'\in T(z')$.

The \emph{Spingarn's partial inverse}~\cite{spi-par.amo83} of a maximal monotone operator
$T:\HH\tos \HH$
with respect to a closed subspace $V$ of $\HH$ is the (maximal monotone) operator 
$T_V:\HH\tos \HH$ whose graph is
\begin{align}
 \label{eq:ema5}
 \mbox{G}(T_V):=\{(z,v)\in \HH\times \HH\;|\; P_V(v)+P_{V^\perp}(z)\in T(P_V(z)+P_{V^\perp}(v))\},
\end{align}
where $P_V$ and $P_{V^\perp}$ stand for the orthogonal projection onto
$V$ and $V^\perp$, respectively.

In the next subsection, we present the convergence rate of the proximal point algorithm for finding zeroes of
strongly monotone inclusions.

\subsection{On the proximal point method for strongly monotone operators}
\label{sec:hpe.str}

In this subsection, we 
consider the problem
\begin{align}
\label{eq:inc.p}
 0\in A(z)
\end{align}
where $A:\HH\tos \HH$ is a 
$\mu$-strongly maximal monotone operator, for some $\mu>0$, i.e., $A$ is maximal monotone and there exists $\mu>0$ such that
\begin{align}
 \inner{z-z'}{v-v'}\geq \mu\norm{z-z'}^2\quad \forall v\in A(z), v'\in A(z').
\end{align}
The proximal point (PP) method is an iterative scheme for finding approximate solutions of \eqref{eq:inc.p} (under the
assumption of maximal monotonicity of $A$). The method traces back to the work of Martinet and was popularized by 
Rockafellar, who formulated and studied some of its inexact variants along with important applications in 
convex programming~\cite{MR0418919,roc-mon.sjco76}.

Among the many inexact variants of the PP method, the hybrid proximal 
extragradient (HPE) method of  Solodov and Svaiter~\cite{sol.sva-hyb.svva99} has been the object of intense study 
by many authors (see, e.g.,~\cite{alv.mon.sva-reg.siam16,bot.cse-hyb.nfao15,cen.mor.yao-hyb.jota10,eck.sil-pra.mp13,he.mon-acc.siam16,ius.sos-pro.opt10,lol.par.sol-cla.jca09,mon.ort.sva-imp.coap14,mon.ort.sva-ada.coap16,MonSva10-2,sol.sva-hyb.svva99,sol.sva-ine.mor00,Sol-Sv:hy.unif}). 
One of its distinctive features is that it allows for
inexact solution of the corresponding subproblems within relative errors tolerances. 
Contrary to the (asymptotic) convergence analysis of the HPE method, which was originally established in \cite{sol.sva-hyb.svva99}, its
iteration-complexity was studied only recently by Monteiro and Svaiter in~\cite{mon.sva-hpe.siam10}. Motivated by the main
results on pointwise and ergodic iteration-complexity which were obtained in the latter reference,  nonasymptotic convergence rates of
the HPE method for solving strongly monotone inclusions were analyzed in~\cite{alv.mon.sva-reg.siam16}.

In this subsection, we specialize the main
result in latter reference, regarding the iteration-complexity of the HPE method for solving strongly monotone inclusions, to the exact  PP method for solving \eqref{eq:inc.p}.
This is motivated by the fact that under certain conditions on the maximal monotone operator $T$, its partial 
inverse $T_V$ -- with respect to a closed subspace $V$ --  is strongly maximal monotone (see Proposition \ref{pr:tvstr}).

\mgap

Next, we present the version of the PP method that we need in this work.

\mgap

\noindent
\fbox{
\addtolength{\linewidth}{-2\fboxsep}%
\addtolength{\linewidth}{-2\fboxrule}%
\begin{minipage}{\linewidth}
\begin{algorithm}
\label{shpe}
{\bf Proximal point (PP) method for solving (\ref{eq:inc.p})}
\end{algorithm}
\begin{itemize}
\item[(0)] Let $z_0\in \HH$ and $\lambda>0$ be given and set $k=1$.
\item[(1)] Compute
       \begin{equation}\label{eq:ec}
         z_k=(\lambda A+I)^{-1}z_{k-1}.    
    \end{equation}
\item[(2)] Let $k\leftarrow k+1$ and go to step 1.
\end{itemize}
\noindent
\end{minipage}
}

\mgap
\mgap

The PP method consists in  successively applying the resolvent $(\lambda A+I)^{-1}$ of the operator $A$ (which is well defined due to the Minty theorem). As we mentioned earlier, there are  more general versions of Algorithm \ref{shpe} in which, in particular, the stepsize parameter $\lambda>0$ is allowed to vary along the iterations and $z_k$  is computed only inexactly in \eqref{eq:ec}.

The following result establishes the linear convergence of the PP method under strong monotonicity assumption on the operator $A$. Although it is a direct consequence of a more general 
result (see \cite[Proposition 2.2]{alv.mon.sva-reg.siam16}), here we present a short proof  for the convenience of the reader.

\begin{proposition}
\label{pr:mmm}
Assume that the operator $A$ is $\mu$-strongly maximal monotone, for some $\mu>0$. Let $\{z_k\}$  be generated by \emph{Algorithm \ref{shpe}} and let $z^*\in \HH$ is the (unique) solution
of \eqref{eq:inc.p}, i.e., $z^*=A^{-1}(0)$.
Then, for all $k\geq 1$, 
\begin{align}
 \label{eq:ema}
 \|z_{k-1}-z_k\|^2 &\le  \left (1-\dfrac{2\lambda\mu}{1+2\lambda\mu}\right)^{k-1}\|z^*-z_0\|^2,\\[2mm]
\label{eq:ema2}
\|z^*-z_k\|^2 &\le \left(1-\dfrac{2\lambda\mu}{1+2\lambda\mu}\right)^{k} \|z^*-z_0\|^2.
\end{align}
\end{proposition}
\begin{proof}
The desired results follow from the following 
\begin{align*}
 \norm{z^*-z_{k-1}}^2-\norm{z^*-z_{k}}^2&=\norm{z_{k-1}-z_k}^2+2\lambda \inner{\lambda^{-1}(z_{k-1}-z_k)}{z_k-z^*}\\
       &\geq  \norm{z_{k-1}-z_k}^2+2\lambda \mu\norm{z_k-z^*}^2\\
      &\geq \dfrac{2\lambda\mu}{1+2\lambda\mu}\norm{z^*-z_{k-1}}^2,
\end{align*}
where we have used that $\lambda^{-1}(z_{k-1}-z_k)\in A(z_k)$, from \eqref{eq:ec},  $0\in A(z^*)$,
$A$ is $\mu$-strongly monotone and $\min\{r^2+ a s^2\,|\,r,s\geq 0, r+s\geq b\geq 0\}=a(a+1)^{-1}b^2$.
%
 %
%
%
%
%
%
\end{proof}

As we mentioned earlier, in the next section we analyze the global (nonasymptotic) convergence rate of the
SPDG algorithm under the assumptions that the operator $T$ in \eqref{eq:pi} is strongly monotone and 
Lipschitz continuous~\cite{mah.qua.tao-prox.siam95}. We will prove, in particular, that under such conditions on $T$,
the partial inverse $T_V$ is strongly monotone as well, in which case the results in \eqref{eq:ema} and \eqref{eq:ema2} can be applied.

\section{On the convergence rate of the SPDG algorithm}
\label{sec:str}

In this section, we consider problem \eqref{eq:pi}, i.e., the problem of finding $x,u\in \HH$ such that
\begin{align}
  \label{eq:inc.v}
  x\in V,\quad  u\in V^\perp\;\;\mbox{and}\;\;u\in T(x), 
\end{align}
where the following  hold:
\begin{itemize}
 \item[A1)] $V$ is a closed vector subspace of $\HH$.
\item[A2)]  $T:\HH\tos \HH$ is (maximal) $\eta$-strongly monotone, i.e., $T$ is maximal monotone and
there exists $\eta>0$ such that
\begin{align}
 \label{eq:str}
 \inner{z-z'}{v-v'}\geq \eta\norm{z-z'}^2 \qquad \forall v\in T(z),v'\in T(z').
\end{align}
\item[A3)] $T:\HH\tos \HH$ is $L$-Lipschitz continuous, i.e.,
there exists $L>0$ such that
\begin{align}
 \label{eq:Lip}
 \norm{v-v'}\leq L\norm{z-z'} \qquad \forall v\in T(z), v'\in T(z').
\end{align}
\end{itemize}

The convergence rate of the  SPDG algorithm under the assumptions A2) and A3) was previously analyzed in \cite[Theorem 4.2]{mah.qua.tao-prox.siam95}. 
Note that \eqref{eq:str} and \eqref{eq:Lip}  imply, in 
particular, that the operator  $T$ is at most single-valued and $L\geq \eta$. The number $L/\eta\geq 1$ is known 
in the literature as the condition number of the
problem \eqref{eq:inc.v}.
The influence of $L/\eta$ as well as of scaling factors in the convergence speed of the SPDG algorithm, specially for solving ill-conditioned problems, was discussed in 
\cite[Section 4]{mah.qua.tao-prox.siam95} for the special case of quadratic programming.

In this section, analogously to the latter reference, we analyze the convergence rate of the SPDG algorithm (Algorithm \ref{spin}) under assumptions A2) and A3) on the
maximal monotone operator $T$. We show that the SPDG algorithm falls in the framework of the PP method (Algorithm \ref{shpe}) for
the (scaled) partial inverse of $T$, which, under assumptions A2) and A3), is shown to be strongly monotone. This contrasts to the
approach adopted in \cite{{mah.qua.tao-prox.siam95}}, which relies on fixed point techniques.
 By showing that the (scaled) partial inverse
of $T$ -- with respect to $V$-- is strongly monotone, we obtain a potentially faster convergence for the SPDG algorithm, when compared
to the one proved in \cite{mah.qua.tao-prox.siam95} by means of fixed point techniques.  Moreover, the convergence rates obtained in this note allows one to measure the convergence speed of the SPDG algorithm on three different measures of approximate solution
to the problem \eqref{eq:inc.v} (see Theorem \ref{th:4} and the remarks right below it). 

Among the above mentioned measures of approximate solution, one of them allows for the study of the iteration-complexity of the
SPDG algorithm in the lines of recent results on the iteration-complexity of the inexact Spingarn's partial inverse method~\cite{cos.alv-spin01} (see \eqref{eq:iter.com}).  
In this regard, one can compute SPDG algorithm's iteration-complexity with respect to the following notion of approximate solution of \eqref{eq:inc.v} (see~\cite{cos.alv-spin01}): for a given
tolerance $\rho>0$, find $x, u\in \HH$  such that
\begin{align}
 \label{eq:inc.v02}
  u\in T(x),\quad \max\left\{\norm{x-P_V(x)}, \gamma \norm{u-P_{V^\perp}(u)}\right\}\leq \rho,
\end{align}
where $\gamma>0$.
For $\rho=0$, criterion \eqref{eq:inc.v02} gives
$x\in V$, $u\in V^{\perp}$ and $u\in T(x)$, i.e., in this case the pair
$(x, u)$ is a solution of \eqref{eq:inc.v}. We mention that criterion \eqref{eq:inc.v02} naturally appears in different settings and
has not been considered in  \cite{{mah.qua.tao-prox.siam95}}.

\mgap

Next, we present the scaled proximal decomposition on the graph of a maximal monotone operator (SPDG) algorithm.

\mgap

\noindent
\fbox{
\addtolength{\linewidth}{-2\fboxsep}%
\addtolength{\linewidth}{-2\fboxrule}%
\begin{minipage}{\linewidth}
\begin{algorithm}
\label{spin}
{\bf SPDG algorithm for solving \bf{(\ref{eq:inc.v})}~\cite[Algorithm 3]{mah.qua.tao-prox.siam95}}
\end{algorithm}
\begin{itemize}
\item[(0)] Let $x_0\in V$, $y_0\in V^\perp$,  $\gamma>0$ be given and set $k=1$.
\item [(1)] Compute
\begin{align}
\label{eq:spin}
 \tilde x_k=(\gamma T+I)^{-1}\left(x_{k-1}+\gamma y_{k-1}\right),\qquad 
u_k=\gamma^{-1}\left(x_{k-1}+\gamma y_{k-1}-\tilde x_k\right).  
\end{align} 
\item[(2)] If $\tilde x_k\in V$ and $u_k\in V^\perp$, then stop. Otherwise, define 
 \begin{align}
  \label{eq:spin2}
   x_k=P_V(\tilde x_k),\quad y_k=P_{V^\perp}(u_k),   
 \end{align}
	set $k\leftarrow k+1$ and go to step 1.
   \end{itemize}
\noindent
\end{minipage}
} 
\mgap

\noindent
{\bf Remarks.}
\begin{itemize}
\item[(i)] 
Algorithm \ref{spin} was originally proposed and studied in \cite{mah.qua.tao-prox.siam95}. When
$\gamma=1$, it reduces to the Spingarn's partial inverse method for solving \eqref{eq:inc.v}. The authors of the latter reference emphasize the importance of introducing the scaling $\gamma>0$ in order to speed up the convergence of the SPDG algorithm, specially when solving ill-conditioned problems.
\item[(ii)] As we mentioned earlier, one of the contributions of this note is to show that similar results (actually potentially better) to
the one obtained in  \cite{mah.qua.tao-prox.siam95} regarding the convergence rate of Algorithm \ref{spin} can be proved by means
of the Spingarn's partial inverse framework, instead of fixed point techniques. 
\end{itemize}

The following result appears (with a different notation) inside the proof of \cite[Theorem 4.2]{mah.qua.tao-prox.siam95}.

\begin{theorem} \emph{(inside the proof of \cite[Theorem 4.2]{mah.qua.tao-prox.siam95})}
 \label{th:mot}
If  $T$ is $\eta$--strongly monotone and Lipschitz continuous with constant $L$, then the convergence of the sequence
$\{(x_k, \gamma y_k)\}$  is linear, in the sense that
\begin{align}
  \label{eq:ema10}
 \norm{x^*-x_k}^2+\gamma^2\norm{u^*-y_k}^2\leq 
  \left(1-\dfrac{2\gamma\eta}{(1+\gamma L)^2}\right)^k d_0^{\,2} \qquad \forall k\geq 1,
\end{align}
where $(x^*,u^*)$ is the (unique) solution of \eqref{eq:inc.v} and
\begin{align}
 \label{eq:def.d0}
 d_{0}:=\sqrt{\|x^*-x_0\|^2+\gamma^2\|u^*-y_0\|^2}.
\end{align}
\end{theorem}

\noindent
{\bf Remarks.}
\begin{itemize}
\item[(i)]
 The optimal convergence speed  is achieved  by letting $\gamma=1/L$  in \eqref{eq:ema10}, in which case
(see p. 461 in \cite{mah.qua.tao-prox.siam95})
\begin{align}
  \label{eq:ema100}
   \norm{x^*-x_k}^2+\gamma^2\norm{u^*-y_k}^2\leq
   \left(1-\dfrac{\eta}{2L}\right)^k d_0^{\,2} \qquad \forall k\geq 1.
\end{align}
\item[(ii)] It follows from \eqref{eq:ema100} that, for a given tolerance $\rho>0$, Algorithm \ref{spin} finds $x,y\in \HH$ such that 
$\norm{x^*-x}^2+\gamma^2\norm{u^*-y}^2\leq \rho$ after performing at most
\begin{align}
 \label{eq:ema500}
 2+\log\left(\dfrac{2L}{2L-\eta}\right)^{-1}\log\left(\dfrac{d_0^{\,2}}{\rho}\right)
\end{align}
iterations. In the third remark after Theorem \ref{th:4}, we show that our approach provides a potentially better upper bound
on the number of iterations needed by the SPDG algorithm to achieve prescribed tolerances, specially on ill-conditioned problems.
\end{itemize}

A direct consequence of the next proposition, in contrast to the reference \cite{mah.qua.tao-prox.siam95}, is that  it is possible to analyze the SPDG algorithm within the original Spingarn's partial inverse framework.
%
We show that under assumptions A2) and A3), the partial inverse operator $T_V$ is strongly monotone.

\begin{proposition}
 \label{pr:tvstr}
 Under the assumptions \emph{A2)} and \emph{A3)} on the maximal monotone operator $T$, its partial inverse
$T_V$ with respect to $V$ is $\mu$-strongly (maximal) monotone with 
\begin{align}
 \label{eq:def.mu}
\mu=\dfrac{\eta}{1+L^2}>0.
\end{align}
\end{proposition}
\begin{proof}
Take $v\in T_V(z)$, $v'\in T_V(z')$ and note that, from \eqref{eq:ema5}, we have
\begin{align}
 \label{eq:inc.ttv}
 P_V(v)+P_{V^\perp}(z)\in T(P_V(z)+P_{V^\perp}(v)),\quad 
P_V(v')+P_{V^\perp}(z')\in T(P_V(z')+P_{V^\perp}(v')),
\end{align}
which, in turn, combined with the assumption A2) and after some direct calculations yields
\begin{align}
 \nonumber
 \inner{z-z'}{v-v'}&=\inner{P_V(z-z')+P_{V^\perp}(z-z')}{P_V(v-v')+P_{V^\perp}(v-v')}\\
 \nonumber
                 &=\inner{P_V(z-z')+P_{V^\perp}(v-v')}{P_V(v-v')+P_{V^\perp}(z-z')}\\
 \nonumber
                 &\geq \eta\norm{P_V(z-z')+P_{V^\perp}(v-v')}^2\\
 \label{eq:tvv2}
                 &=\eta\left(\norm{P_V(z-z')}^2+\norm{P_{V^\perp}(v-v')}^2\right)\\
\label{eq:tvv22}
								 &\geq \eta \norm{P_V(z-z')}^2.
\end{align}
On the other hand, assumption A3) and \eqref{eq:inc.ttv} imply
\[
 \norm{[P_V(v)+P_{V^\perp}(z)]-[P_V(v')+P_{V^\perp}(z')]}\leq 
 L\norm{[P_V(z)+P_{V^\perp}(v)]-[P_V(z')+P_{V^\perp}(v')]},
\]
which, in particular, gives
\begin{align}
 \nonumber
 \norm{P_V(z-z')}^2+\norm{P_{V^\perp}(v-v')}^2&\geq 
\dfrac{1}{L^2}\left(\norm{P_{V}(v-v')}^2+\norm{P_{V^\perp}(z-z')}^2\right)\\
 \label{eq:tvv33}
     &\geq \dfrac{1}{L^2}\norm{P_{V^\perp}(z-z')}^2.
\end{align}
Using \eqref{eq:tvv22} and combining \eqref{eq:tvv2} and \eqref{eq:tvv33} we find, respectively,
\[
 \inner{z-z'}{v-v'}\geq \eta \norm{P_{V}(z-z')}^2,\quad
 L^2\inner{z-z'}{v-v'}\geq \eta \norm{P_{V^\perp}(z-z')}^2.
\]
The desired result now follows by adding the above inequalities
and by using the definition of $\mu>0$ in \eqref{eq:def.mu}. 
\end{proof}

Next, we show that Algorithm \ref{spin} is a special instance of Algorithm \ref{shpe}.


\begin{proposition}
 \label{pr:pi.e.hpe}
 Let $\{x_k\}$  and $\{y_k\}$ be generated by \emph{Algorithm \ref{spin}} and define
  \begin{align}
	  \label{eq:pi.e.hpe}
		 z_k=x_k+\gamma y_k\qquad \forall k\geq 0.
		 \end{align}
Then, for all $k\geq 1$,
\begin{align}
 \label{eq:pi.e.hpe012}
  &z_{k-1}-z_k=P_V(\gamma u_k)+P_{V^\perp}(\tilde x_k),\\
 \label{eq:pi.e.hpe02}
 & z_{k}=\left((\gamma T)_V+I\right)^{-1}z_{k-1}.
\end{align}
As a consequence of \eqref{eq:pi.e.hpe02}, we have that \emph{Algorithm \ref{spin}} is a special instance of  \emph{Algorithm \ref{shpe}} with $\lambda=1$ for solving \eqref{eq:inc.p} with $A=(\gamma T)_V$.
\end{proposition}
\begin{proof}
 Using \eqref{eq:spin},  we obtain  $\gamma u_k\in (\gamma T)(\tilde x_k)$ and, as a consequence, from
\eqref{eq:ema5}, we have
\begin{align}
 \label{eq:ema20}
 P_V(\gamma u_k)+P_{V^\perp}(\tilde x_k)\in (\gamma T)_V(P_V(\tilde x_k)+P_{V^\perp}(\gamma u_k)).
\end{align}
From the second identity in \eqref{eq:spin}, we have 
$\gamma u_k+\tilde x_k=x_{k-1}+\gamma y_{k-1}$, which combined with \eqref{eq:spin2} gives
\[
 P_V(\gamma u_k)=x_{k-1}-x_k,\qquad P_{V^\perp}(\tilde x_k)=\gamma(y_{k-1}-y_k),
\]
which, in turn, is  equivalent to \eqref{eq:pi.e.hpe012}.
Using \eqref{eq:ema20}, \eqref{eq:pi.e.hpe012}, \eqref{eq:pi.e.hpe} and \eqref{eq:spin2}, we find 
$z_{k-1}-z_k\in (\gamma T)_V(z_k)$, which is clearly equivalent to \eqref{eq:pi.e.hpe02}.
The last statement of the proposition follows directly from \eqref{eq:pi.e.hpe02} and Algorithm \ref{shpe}'s definition.
\end{proof}

In the next theorem, we present the main contribution of this note, namely, convergence rates for the SPDG algorithm obtained
within the original Spingarn's partial inverse framework.

\begin{theorem} 
\label{th:4}
Let $\{x_k\}$, $\{y_k\}$, $\{\tilde x_k\}$ and $\{u_k\}$ be generated by \emph{Algorithm \ref{spin}}, let 
$(x^*,u^*)$ be the (unique) solution of  \eqref{eq:inc.v}
and let $d_0$ be as in \eqref{eq:def.d0}.
Then, for all $k\geq 1$,
\begin{align}
 \label{eq:ema13}
 \norm{x_{k-1}-x_k}^2+\gamma^2\norm{y_{k-1}-y_k}^2 &\leq 
 \left(1-\dfrac{2\gamma \eta}{(1+\gamma L)^2-2\gamma(L-\eta)}\right)^{k-1} d_{0}^{\,2},\\
 \label{eq:ema11}
 \norm{\tilde x_k-P_V(\tilde x_k)}^2+\gamma^2\norm{u_k-P_{V^\perp}(u_k)}^2 &\leq 
 \left(1-\dfrac{2\gamma \eta}{(1+\gamma L)^2-2\gamma(L-\eta)}\right)^{k-1} d_{0}^{\,2},\\
\label{eq:ema12}
 \norm{x^*-x_k}^2+\gamma^2\norm{u^*-y_k}^2 &\leq 
 \left(1-\dfrac{2\gamma \eta}{(1+\gamma L)^2-2\gamma(L-\eta)}\right)^{k} d_{0}^{\,2}.
\end{align}
\end{theorem}
\begin{proof}
First, note that, from \eqref{eq:ema5}, $(x^*,u^*)$ is a solution  of \eqref{eq:inc.v} if and only if $x^*+\gamma u^*=:z^*\in (\gamma T)_V^{-1}(0)$.
Using the last statement in Proposition \ref{pr:pi.e.hpe}, Proposition \ref{pr:tvstr} to the operator $\gamma T$ (which is
$(\gamma\eta)$-strongly monotone and $(\gamma L)$-Lipschitz continuous) and Proposition \ref{pr:mmm}, we conclude that
the inequalities \eqref{eq:ema} and \eqref{eq:ema2} hold with $z^*$ as above, $\lambda=1$, $z_k$ as in \eqref{eq:pi.e.hpe} and
\[
 \mu=\dfrac{\gamma \eta}{1+(\gamma L)^2}.
\]
Direct calculations yield (recall that $\lambda=1$)
%
\[
 \dfrac{2 \mu}{1+2\mu}=
\dfrac{2\gamma \eta}{(1+\gamma L)^2-2\gamma(L-\eta)}.
\]
Hence, \eqref{eq:ema13} and \eqref{eq:ema12} follow from  \eqref{eq:ema} and \eqref{eq:ema2}, respectively. To finish
the proof, it remains to prove \eqref{eq:ema11}. To this end, note that it follows from \eqref{eq:pi.e.hpe}, 
\eqref{eq:pi.e.hpe012}, \eqref{eq:ema13} and the facts that $P_{V^\perp}=I-P_V$ and $P_{V}=I-P_{V^\perp}$.
\end{proof}

\noindent
{\bf Remarks.}
\begin{enumerate}
\item[(i)] Analogously to first remark after Theorem \ref{th:mot}, one can easily verify that
$\gamma=1/L$ provides the best convergence speed in  \eqref{eq:ema13}--\eqref{eq:ema12}, in which case we find, respectively,

\begin{align}
 \label{eq:ema133}
 \norm{x_{k-1}-x_k}^2+\gamma^2\norm{y_{k-1}-y_k}^2 &\leq 
 \left(1-\dfrac{\eta}{\eta+L}\right)^{k-1} d_{0}^{\,2},\\
 \label{eq:ema111}
 \norm{\tilde x_k-P_V(\tilde x_k)}^2+\gamma^2\norm{u_k-P_{V^\perp}(u_k)}^2 &\leq 
 \left(1-\dfrac{\eta}{\eta+L}\right)^{k-1} d_{0}^{\,2},\\
\label{eq:ema122}
 \norm{x^*-x_k}^2+\gamma^2\norm{u^*-y_k}^2 &\leq 
\left(1-\dfrac{\eta}{\eta+L}\right)^{k} d_{0}^{\,2}.
\end{align}
\item[(ii)] Since $L\geq \eta$, it follows that the convergence speed obtained in \eqref{eq:ema133}--\eqref{eq:ema122} 
is potentially better (specially on ill-conditioned problems) than the optimal one
obtained for the SPDG algorithm in \cite{mah.qua.tao-prox.siam95} via fixed point techniques, namely \eqref{eq:ema100}. The same remark applies to the rates in \eqref{eq:ema13}--\eqref{eq:ema12}, when compared to the corresponding one in \eqref{eq:ema10}.
See Figure 1.
\item[(iii)] Note that \eqref{eq:ema111} (resp. \eqref{eq:ema122}) imply that, for a given tolerance $\rho>0$, 
the SPDG algorithm finds a pair $(x,u)$ (resp. $(x,\gamma y)$) satisfying the termination criterion \eqref{eq:inc.v02} 
(resp. $\norm{x^*-x}^2+\gamma^2\norm{u^*-y}^2\leq \rho$) after performing no more than
\begin{align}
 \label{eq:iter.com}
 2+\log\left(\dfrac{\eta+L}{L}\right)^{-1}\log\left(\dfrac{d_{0}^{\,2}}{\rho}\right)
\end{align}
iterations.
\item[(iv)] By taking $L=57$ and $\eta=9$  (see p. 462 in \cite{mah.qua.tao-prox.siam95}), we find 
$\log\left(\dfrac{2L}{2L-\eta}\right)^{-1}\approx 13$  and 
$\log\left(\dfrac{\eta+L}{L}\right)^{-1}\approx 7$  in \eqref{eq:ema500} and \eqref{eq:iter.com}, respectively. This shows that
the upper bound \eqref{eq:iter.com}, obtained in this note, is more accurate than the corresponding one \eqref{eq:ema500}, obtained
in \cite{mah.qua.tao-prox.siam95}.
\end{enumerate}

\begin{figure}[H]
\centering
\includegraphics[scale=.6]{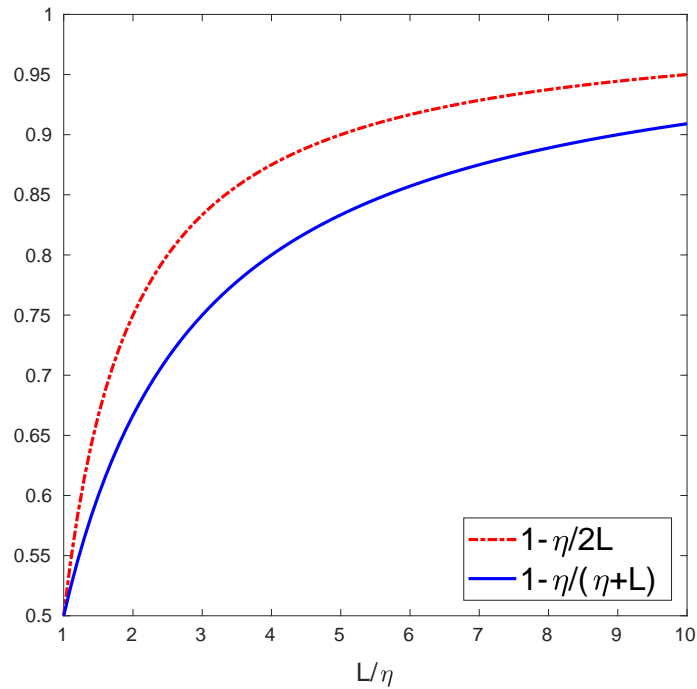}
\caption{Solid line: see the convergence rates \eqref{eq:ema133}--\eqref{eq:ema122}; dotted line: see \eqref{eq:ema100}.}
\label{figure1}
\end{figure}

\section{Concluding remarks}
 \label{sec:cr}

In this note, we proved that the SPDG algorithm of  Mahey, Oualibouch and Tao  
introduced in \cite{mah.qua.tao-prox.siam95}, can alternatively be analyzed within the Spingarn's partial inverse framework, instead of the
fixed point approach proposed in the latter reference. This traces back to the 1983 Spingarn's paper, where, among other contributions, he introduced and analyzed the partial inverse method. We simply proved that under the assumptions of \cite{mah.qua.tao-prox.siam95}, namely strong monotonicity and Lipschitz continuity, the Spingarn's partial inverse of the underlying maximal monotone operator
is strongly monotone as well. This allowed us to employ recent developments in the convergence analysis and iteration-complexity
of proximal point type methods for strongly monotone operators. By doing this, we additionally obtained a potentially better convergence speed for the SPDG algorithm as well as a better upper bound on the number of iterations needed to achieve prescribed tolerances.


\bibliographystyle{plain}

\end{document}